\documentclass[article]{ij4uq}              
\newcommand{\y}{\mathbf{y}}
\newcommand{\x}{\mathbf{x}}
\newcommand{\yp}{\mathbf{y}'}
\newcommand{\vn}[1]{\lVert#1\Vert}

\fancypagestyle{plain}{%
  \fancyhf{}
  \fancyhead[R]{\small \phantom{{\it International Journal for Uncertainty Quantification}, 1(1):xxx--xxx, \myyear\today}}
  \fancyfoot[R]{\vspace*{10pt}\small\bf\thepage }
  \fancyfoot[L]{\fottitle}
  }

\begin{document}

\volume{Volume 1, Number 1, \myyear\today}
\title{A multi-fidelity stochastic collocation\\ method for parabolic PDEs with
random\\ input data}
\titlehead{Multi-fidelity Stochastic Collocation}
\authorhead{Raissi \& Seshaiyer}
\author{Maziar Raissi}
\corrauthor{Padmanabhan Seshaiyer}
\corremail{pseshaiy@gmu.edu}
\corrurl{http://math.gmu.edu/~pseshaiy/}
\address{Department of Mathematical Sciences, George Mason University, 4400 University Drive, MS: 3F2, Planetary Hall, Fairfax, VA 22030, USA}

\dataO{\mydate\today}
\dataF{\mydate\today}

\abstract{Over the last few years there have been dramatic advances in our understanding of mathematical and computational models of complex systems in the presence of uncertainty. This has led to a growth in the area of uncertainty quantification as well as the need to develop efficient, scalable, stable and convergent computational methods for solving differential equations with random inputs. Stochastic Galerkin methods based on polynomial chaos expansions have shown superiority to other non-sampling and many sampling techniques. However, for complicated governing equations numerical implementations of stochastic Galerkin methods can become non-trivial. On the other hand, Monte Carlo and other traditional sampling methods, are straightforward to implement. However, they do not offer as fast convergence rates as stochastic Galerkin. Other numerical approaches are the stochastic collocation (SC) methods, which inherit both, the ease of implementation of Monte Carlo and the robustness of stochastic Galerkin to a great deal. In this work we propose a novel enhancement to stochastic collocation methods using deterministic model reduction techniques. Linear parabolic partial differential equations with random forcing terms are analysed. The input data are assumed to be represented by a finite number of random variables. A rigorous convergence analysis, supported by numerical results, shows that the proposed technique is not only reliable and robust but also efficient.
}

\keywords{collocation, stochastic partial differential equations, sparse grid, smolyak algorithm, finite element, proper orthogonal decomposition, multi-fidelity}

\maketitle

\section{Introduction}

\label{sec:Introduction}
The effectiveness of stochastic partial differential equations (SPDEs) in modelling complicated phenomena is a well-known fact. One can name wave propagation \cite{papanicolaou1971wave}, diffusion through heterogeneous random media \cite{papanicolaou1995diffusion}, randomly forced Burgers and NavierStokes equations (see e.g \cite{bensoussan1973equations,da2003ergodicity,khanin1997probability,mikulevicius2004stochastic} and the references therein) as a couple of examples. Currently, Monte Carlo is by far the most widely used tool in simulating models driven by SPDEs. However, Monte Carlo simulations are generally very expensive. To meet this concern, methods based on the Fourier analysis with respect to the Gaussian (rather than Lebesgue) measure, have been investigated in recent decades. More specifically, Cameron--Martin version of the Wiener Chaos expansion (see, e.g. \cite{cameron1947orthogonal,hida1993white} and the references therein) is among the earlier efforts. Sometimes, the Wiener Chaos expansion (WCE for short) is also referred to as the Hermite polynomial chaos expansion. The term polynomial chaos was coined by Nobert Wiener \cite{wiener1938homogeneous}. In Wieners work, Hermite polynomials served as an orthogonal basis. The validity of the approach was then proved in \cite{cameron1947orthogonal}. There is a long history of using WCE as well as other polynomial chaos expansions in problems in physics and engineering. See, e.g. \cite{crow1970relationship,orszag1967dynamical,chorin1971hermite,chorin1974gaussian}, etc. Applications of the polynomial chaos to stochastic PDEs considered in the literature typically deal with stochastic input generated by a finite number of random variables (see, e.g. \cite{sakamoto2002simulation,ghanem2003stochastic,xiu2003modeling,zhang2004efficient}). This assumption is usually introduced either directly or via a representation of the stochastic input by a truncated Karhunen--Lo\`{e}ve expansion. Stochastic finite element methods based on the Karhunen--Lo\`{e}ve expansion and Hermite polynomial chaos expansion \cite{ghanem2003stochastic,sakamoto2002simulation} have been developed by Ghanem and other authors. Karniadakis et al. generalized this idea to other types of randomness and polynomials \cite{jardak2002spectral,xiu2003modeling,xiu2002wiener}. The stochastic finite element procedure often results in a set of coupled deterministic equations which requires additional effort to be solved. To resolve this issue, stochastic collocation (SC) method was introduced. In this method one repeatedly executes an established deterministic code on a prescribed node in the random space defined by the random inputs. The idea can be found in early works such as \cite{mathelin2003stochastic,tatang1997efficient}. In these works mostly tensor products of one-dimensional nodes (e.g., Gauss quadrature) are employed. Tensor product construction despite making mathematical analysis more accessible (cf. \cite{babuvska2007stochastic}) leads to the curse of dimensionality since the total number of nodes grows exponentially fast as the number of random parameters increases. In recent years we are experiencing a surge of interest in the high-order stochastic collocation approach following \cite{xiu2005high}. The use of sparse grids from multivariate interpolation analysis, is a distinct feature of the work in \cite{xiu2005high}. A sparse grid, being a subset of the full tensor grid, can retain many of the accuracy properties of the tensor grid. While keeping high-order accuracy, it can significantly reduce the number of nodes in higher random dimensions. Further reduction in the number of nodes was pursued in \cite{agarwal2009domain,ma2009adaptive,nobile2008sparse,nobile2008anisotropic}. Applications of these numerical methods take a wide range. Here we mention some of the more representative works. It includes Burgers’ equation \cite{hou2006wiener,xiu2004supersensitivity}, fluid dynamics \cite{knio2006uncertainty,knio2001stochastic,le2002stochastic,lin2007stochastic,xiu2003modeling}, flow-structure interactions \cite{xiu2002stochastic}, hyperbolic problems \cite{chen2005uncertainty,gottlieb2008galerkin,lin2006predicting}, model construction and reduction \cite{doostan2007stochastic,ghanem2005identification,ghanem2006construction}, random domains with rough boundaries \cite{canuto2007fictitious,lin2007random,tartakovsky2006stochastic,xiu2006numerical}, etc.

Along with an attempt to reduce the number of nodes used by sparse grid stochastic collocation, one can try to employ more efficient deterministic algorithms. The current trend is to repeatedly execute a full-scale underlying deterministic simulation on prescribed nodes in the random space. However, model reduction techniques can be employed to create a computationally cheap deterministic algorithm that can be used for most of the grid points. This way we can limit the employment of an established while computationally expensive algorithm to only a relatively small number of points. A similar method is being used by K. Willcox and her team but in the context of optimization \cite{robinson2008surrogate}. ``Multifidelity", which we also adopt, is the term they employed in their work. Reduced order modelling, using proper orthogonal decompositions (POD) along with Galerkin projection, for fluid flows has seen extensive applications studied in \cite{sirovich1987turbulence,chambers1988karhunen,holmes1998turbulence,fahl2001computation,iollo2000stability,
kunisch2001galerkin,kunisch2002galerkin,henri2002stability,rowley2004model,camphouse2005boundary}. Proper orthogonal decomposition (POD) was introduce in Pearson \cite{pearson1901liii} and Hotelling \cite{hotelling1933analysis}. Since the work of Pearson and Hotelling, many have studied or used POD in a range of fields such as oceanography \cite{bjornsson1997manual}, fluid mechanics \cite{sirovich1987turbulence,holmes1998turbulence}, system feedback control \cite{ravindran2000reduced,atwell2001reduced,atwell2001proper,kepler2002reduced,atwell2004reduced,lee2005reduced}, and system modeling \cite{fahl2001computation,kunisch2002galerkin,rowley2004model,henri2005convergence}. 
In this work we analyse linear parabolic partial differential equations with random forcing terms. We propose a novel method which dramatically decreases the computational cost. The idea of the method is very simple. For each point of the stochastic parameter domain we search to see if the resulting deterministic problem is already solved for a sufficiently close problem. If yes, we use the solution to the nearby problem to create POD basis functions and we employ POD-Galerkin method to solve the original problem. We provide a rigorous convergence analysis for our proposed method. Finally, it is shown by numerical examples that the results of numerical computation are consistent with theoretical conclusions.

\section{Problem definition}
\label{sec:ProblemDef} Let $D\subset\mathbb{R}^2$ be a bounded, connected and polygonal domain and $(\Omega,\mathcal{F},P)$ denote a complete probability space with sample space $\Omega$, which corresponds to the set of all possible outcomes. $\mathcal{F}$ is the $\sigma$-algebra of events, and
$P:\mathcal{F}\rightarrow[0,1]$ is the probability measure. In this section, we consider the stochastic linear parabolic initial-boundary value problem: find a random field $u:[0,T]\times\overline{D}\times\Omega \rightarrow \mathbb{R}$, such that $P$-almost surely the following equations hold:
\begin{eqnarray}
    \label{eqn:SParabolic}
  \partial_t u(t,\mathbf{x},\omega)-\Delta u(t,\mathbf{x},\omega)=f(t,\mathbf{x},\omega) &\text{~~~~~in~~}& (0,T]\times D\times\Omega, \nonumber\\
  u(t,\mathbf{x},\omega) = 0 &\text{~~~~~on~~}& (0,T]\times \partial D\times\Omega, \\
  u(0,\mathbf{x},\omega) = 0 &\text{~~~~~on~~}& D\times\Omega\nonumber.
\end{eqnarray}
In order to guarantee the existence and uniqueness of the solution of (\ref{eqn:SParabolic}), we assume that the random forcing field $f: [0,T] \times \overline{D}\times\Omega\ni (t,\mathbf{x},\omega) \mapsto f(t,\mathbf{x},\omega) \in \mathbb{R}$ satisfies:
\begin{eqnarray}
\label{ass:f}
\int_0^T\int_{D}\mathbb{E}[f^2]d\mathbf{x}dt < +\infty \Longleftrightarrow \int_0^T\int_D f^2(t,\mathbf{x},\omega)d\mathbf{x}dt < +\infty ~~~~~~~~ P\text{-a.e. in ~} \Omega.
\end{eqnarray}
Following \cite{babuvska2007stochastic} and inspired by the truncated KL expansion \cite{loeve1977probability}, we make the assumption that the random field $f$ depends on a finite number of independent random variables. More specifically,
\begin{eqnarray}
\label{ass:FiniteNoise}
f(t,\mathbf{x},\omega) = f(t,\mathbf{x},\mathbf{y}(\omega)) & ~~~~~\text{on}& ~~[0,T]\times\overline{D}\times\Omega,
\end{eqnarray}
where $\mathbf{y}(\omega) = (y_1(\omega),\ldots,y_r(\omega))$ and $r\in\mathbb{N}_+$. Lets define the space,
\[L_P^2(\Omega) := \{\mathbf{y} = (y_1,y_2,\ldots,y_r)^T:\sum_{n=1}^{r}\int_{\Omega}|y_n(\omega)|^2dP(\omega) < \infty\},\]
where $\mathbf{y}$ denotes an $r$-dimensional random vector over $(\Omega,\mathcal{F},P)$. We also define the Hilbert space,
\[V:= L^2(0,T;H_0^1(D))\otimes L_P^2(\Omega),\]
with the inner product $(.,.)_V:V\times V\rightarrow\mathbb{R}$ given by:
\[(u.v)_V = \int_0^T\int_D\mathbb{E}[\nabla u(t,\mathbf{x},\omega).\nabla v(t,\mathbf{x},\omega)]d\mathbf{x}dt.\]
A function $u \in V$ is called a \textit{weak solution} of problem (\ref{eqn:SParabolic}) if:
\begin{eqnarray}
\label{eqn:weak}
\int_{D}\mathbb{E}[\partial_t u v]d\mathbf{x} + \int_D\mathbb{E}[\nabla u . \nabla v]d\mathbf{x} = \int_D\mathbb{E}[f v]d\mathbf{x}, ~~~~~\forall v \in H_0^1(D)\otimes L_P^2(\Omega) \text{~and~} \forall t \in (0,T],
\end{eqnarray}
and $P$-almost surely $u(0,\mathbf{x},\omega) =0$. The existence and uniqueness of the solution of problem (\ref{eqn:weak}) is a direct consequence of assumption (\ref{ass:f}) on $f$; see \cite{evanspartial}.\\

Let $\Gamma_n = y_n(\Omega)$ denote the image of the random variable $y_n$, for $n = 1,\ldots,r$, and $\Gamma = \prod_{n=1}^r \Gamma_n$. We also assume that the distribution measure of $\mathbf{y}(\omega)$ is absolutely continuous with respect to the Lebesgue measure. Thus, there exists a joint density function $\rho:\Gamma \rightarrow \mathbb{R}_+$ for $\mathbf{y} = (y_1,\ldots,y_r)$. Hence, we can use $(\Gamma,\mathcal{B}^r,\rho d\mathbf{y})$ instead of $(\Omega,\mathcal{F},P)$, where $\mathcal{B}^r$ is the $r$-dimensional Borel space. Analogous to the definitions of $L_P^2(\Omega)$ and $V$ we can define
\[L_\rho^2(\Gamma) := \{\mathbf{y}\in \Gamma:\int_{\Gamma}||\mathbf{y}||^2\rho d\mathbf{y} < \infty\},\]
and
\[V_\rho = L^2(0,T;H_0^1(D))\otimes L_\rho^2(\Gamma),\]
with inner product
\[(u.v)_{V_\rho} = \int_{\Gamma}(u(\mathbf{y}),v(\mathbf{y}))_{L^2(0,T;H_0^1(D))}\rho d\mathbf{y}.\]
where
\[(u(\mathbf{y}),v(\mathbf{y}))_{L^2(0,T;H_0^1(D))} = \int_0^T\int_D\nabla u(t,\mathbf{x},\mathbf{y}).\nabla v(t,\mathbf{x},\mathbf{y})d\mathbf{x}dt.\]
The weak solution $u\in V$ of problem (\ref{eqn:SParabolic}), using the finite dimensional noise assumption (\ref{ass:FiniteNoise}), is of the form $u(t,\mathbf{x},\omega)=u(t,\mathbf{x},y_1(\omega),\ldots,y_r(\omega))$. Therefore, the weak formulation (\ref{eqn:weak}) can be equivalently expressed as finding $u\in V_\rho$ such that $\rho$-almost everywhere in $\Gamma$, $u(0,\mathbf{x},\mathbf{y}) =0$, and
\begin{eqnarray}
\label{eqn:weak2}
\int _{\Gamma} \int_D \partial_t u v d\mathbf{x}\rho d\mathbf{y} + \int_\Gamma\int_D \nabla u .\nabla vd\mathbf{x}\rho d\mathbf{y}= \int_\Gamma\int_D f v d\mathbf{x}\rho d\mathbf{y}, ~~~~~\forall v \in H_0^1(D)\otimes L_\rho^2(\Gamma) \text{~and~} \forall t \in (0,T].
\end{eqnarray}
For each fixed $t\in (0,T]$, the solution $u$ to (\ref{eqn:weak2}) can be viewed as a mapping $u:\Gamma \rightarrow H_0^1(D)$. In order to emphasize the dependence on the variable $\mathbf{y}$, we use the notations $u(\mathbf{y})$ and $f(\mathbf{y})$. Hence, we achieve the following equivalent settings: find $u(\mathbf{y})\in H_0^1(D)$ such that $\rho$-almost everywhere in $\Gamma$, $u(0,\mathbf{x},\mathbf{y}) = 0$, and
\begin{eqnarray}
\label{eqn:weak3}
\int_D \partial_t u(\mathbf{y}) v d\mathbf{x} + \int_D \nabla u(\mathbf{y}) .\nabla v d\mathbf{x}= \int_D f(\mathbf{y}) v d\mathbf{x}, ~~~~~\forall v \in H_0^1(D) \text{~and~} \forall t \in (0,T], ~\rho\text{-a.e. in } \Gamma.
\end{eqnarray}
Note that there may exist a $\rho d\mathbf{y}$-zero measure set $\mathcal{N}_\rho \subset \Gamma$ in which (\ref{eqn:weak3}) is not satisfied. Therefore, from a computational perspective, if a point $\mathbf{y}\in \mathcal{N}_\rho$ is chosen, the resulting solution of (\ref{eqn:weak3}) is not the true solution of the original equation. However, the computation of the moments of the solution does not suffer from this disadvantage.

\section{Multi-fidelity Collocation method}
\label{sec:MFCollocationMethod}
In this section, we apply our multi-fidelity stochastic collocation method to the weak form (\ref{eqn:weak3}). Let $V_{\rho,h}$ be a finite dimensional subspace of $V_\rho$ given by $V_{\rho,h} = L^2(0,T;H_h(D))\otimes \mathcal{P}_\mathbf{p}(\Gamma)$, where $H_h(D) \subset H_0^1(D)$ is a standard finite element space and $\mathcal{P}_\mathbf{p}(\Gamma)\subset L_\rho^2(\Gamma)$ is the span of tensor product polynomials with degree
at most $\mathbf{p} = (p_1,\ldots,p_r)$. The goal is to find a numerical approximation to the solution of (\ref{eqn:weak3}) in the finite dimensional subspace $V_{\rho,h}$. Choose $\eta > 0$ to be a small real number. The procedure for solving (\ref{eqn:weak3}) is divided into two parts:
\begin{enumerate}
\item Fix $\mathbf{y}\in\Gamma$, and search the $\eta$-neighbourhood $B_\eta(\mathbf{y})\subset\Gamma$ of $\mathbf{y}$. If problem (\ref{eqn:weak3}) is not already solved for any nearby problem with $\mathbf{y}' \in B_\eta(\mathbf{y})$, solve problem (\ref{eqn:weak3}) using a regular backward Euler finite element method at $\mathbf{y}$ and let $\mathbf{y}' = \mathbf{y}$. In contrast, if equation (\ref{eqn:weak3}) is already solved for some points in $B_\eta(\mathbf{y})$, choose one of them and call it $\mathbf{y}'$. In either case, use the solution at $\mathbf{y}'\in B_\eta(\mathbf{y})$ to find a small number $d\in\mathbb{N}_+$ of suitable orthonormal basis functions $\{\psi_j(\mathbf{y}')\}_{j=1}^d\subset H_h(D)$ using proper orthogonal decomposition (POD) method. Now use Galerkin projection on to the subspace $X^d(\mathbf{y}') = \text{span}\{\psi_j(\mathbf{y}')\}_{j=1}^d$ to find
\[\{u_d^m(\mathbf{y})\}_{m=1}^N\subset X^d(\mathbf{y}') \subset H_h(D),\]
such that
\begin{eqnarray}
\label{part1}
(u_d^m,v_d)+k(\nabla u_d^m,\nabla v_d) = k(f^m(\mathbf{y}),v_d)+(u_d^{m-1},v_d),~~~\forall v_d\in X^d(\mathbf{y}'), ~~~m=1,\ldots,N,
\end{eqnarray}
and $u_d^0=0$, where $N\in\mathbb{N_+}$ is the number of time steps, and $k = T/N$ denotes the time step increments. It is worth mentioning that $(.,.)$ denotes the $L^2$-inner product. Note that we are employing a backward Euler scheme to discretize time.
\item Collocate (\ref{part1}) on zeros of suitable orthogonal polynomials and build the interpolated discrete solution
\begin{eqnarray}
\label{part2}
\{u_{d,\mathbf{p}}^m\}_{m=1}^N\subset H_h(D)\otimes \mathcal{P}_\mathbf{p}(\Gamma),
\end{eqnarray}
using
\begin{eqnarray}
\label{CollocationScheme}
u_{d,\mathbf{p}}^m(x,\mathbf{y}) = \mathcal{I}_\mathbf{p}u_d^m(x,\mathbf{y}) = \sum_{j_1=1}^{p_1+1}\cdots\sum_{j_r=1}^{p_r+1}u_d^m(x,y_{j_1},\ldots,y_{j_r})(l_{j_1}(\mathbf{y})\otimes\cdots\otimes l_{j_r}(\mathbf{y})), ~~~ m=1,\ldots,N,
\end{eqnarray}
\end{enumerate}
where the functions $\{l_{j_k}\}_{k=1}^r$ can be taken as Lagrange polynomials.  Using this formula, as described in \cite{babuvska2007stochastic}, mean value and variance of $u$ can also be easily approximated.
\subsection{Proper orthogonal decomposition}
In this section, we choose a fixed $\yp \in B_\eta(\y) \subset \Gamma$ and drop the dependence of equation (\ref{eqn:weak3}) on $\yp$, for notational conveniences. Therefore, we consider the problem of finding $w\in H_0^1(D)$ such that:
\begin{eqnarray}
\label{eqn:DeterministicHeat}
(w_t,v) + (\nabla w, \nabla v) = (g,v),~~~~\forall v \in H_0^1(D),
\end{eqnarray}
and $w(\mathbf{x},0) = 0$, for all $\mathbf{x}\in D$. Note that $g = f(\yp)$. Let $t_m=mk, k = 0,\ldots,N$, where $k$ denotes the time step increments. Assume $\mathfrak{T}_h$ to be a uniformly regular family of triangulation of $\overline{D}$ (see \cite{vidar1997galerkin,ciarlet1978finite}). The finite element space is taken as
\[H_h(D) = \{v_h \in H_0^1(D)\cap C^0(D): v_h|_K\in P_s(K),~~\forall K\in\mathfrak{T}_h\},\]
where $s\in\mathbb{N}_+$ and $P_s(K)$ is the space of polynomials of degree $\leq s$ on $K$. Write $w^m(\mathbf{x}) = w(\mathbf{x},t_m)$, and let $w_h^m$ denote the fully discrete approximation of $w$ resulting from solving the problem of finding $w_h^m\in H_h(D)$ such that $w_h^0(\mathbf{x})=0$ and for $m=1,\ldots,N$,
\begin{eqnarray}
\label{eqn:DeterministicHeatFiniteElement}
(w_h^m,v_h)+k(\nabla w_h^m,\nabla v_h) = k(g^m,v_h)+(w_h^{m-1},v_h),~~~\forall v_h\in H_h(D), ~~~m=1,\ldots,N.
\end{eqnarray}
It is easy to prove that problem (\ref{eqn:DeterministicHeatFiniteElement}) has a unique solution $w_h^m\in H_h(D)$, provided that $g^m\in L^2(D)$ (see \cite{vidar1997galerkin}). One can also show that if $w_t \in H^{s+1}(D)$ and $w_{tt}\in L^2(D)$, the following error estimates hold:
\begin{eqnarray}
\label{FEErrorEstimates}
||w^m -w_h^m||_0\leq Ch^{s+1}\int_0^{t_m}||w_t||_{s+1}dt + Ck\int_0^{t_m}||w_{tt}||_0dt, ~~~m=1,\ldots,N,
\end{eqnarray}
where $||.||_s$ denotes the $H^s(D)$-norm and $C$ indicates a positive constant independent of the spatial and temporal mesh sizes, possibly different at distinct occurrences.\\

For the so-called snapshots $U_i := w_h^{m_i} \in H_h(D)$, $i=1,\ldots,\ell$, where $1 \leq m_1 < m_2 < \cdots < m_\ell\leq N$, let
\[\mathcal{V} = \text{span}\{U_1,\ldots,U_\ell\}.\]
Assume at least one of $U_i$ is non-zero, and let $\{\psi_j\}_{j=1}^l$ be an orthonormal basis of $\mathcal{V}$ with $l = \text{dim}\mathcal{V}$. Therefore, for each $U_i\in \mathcal{V}$ we will have:
\begin{eqnarray}
U_i = \sum_{j=1}^l(U_i,\psi_j)_{H_0^1(D)}\psi_j,
\end{eqnarray}
where $(U_i,\psi_j)_{H_0^1(D)} = (\nabla u_h^{m_i},\nabla \psi_j)$.
\begin{defi}
The POD method consists of finding an orthonormal basis $\psi_j ~(j=1,2,\ldots,d)$ such that for every $d = 1,\ldots,l$, the following problem is solved
\begin{eqnarray}
\label{minPOD}
\min_{\{\psi_j\}_{j=1}^d} \frac{1}{\ell}\sum_{i=1}^\ell \lVert U_i - \sum_{j=1}^d(U_i,\psi_j)_{H_0^1(D)}\psi_j\rVert_{H_0^1(D)}^2.
\end{eqnarray}
A solution $\{\psi_j\}_{j=1}^d$ of this minimization problem is known as a POD basis of rank $d$.
\end{defi}
Let us introduce the correlation matrix $K = (K_{ij})_{i,j=1}^\ell \in \mathbb{R}^{\ell\times\ell}$ given by
\begin{eqnarray}
K_{ij} = \frac{1}{\ell}(U_i,U_j)_{H_0^1(D)}.
\end{eqnarray}
The following proposition (see \cite{sirovich1987turbulence,kunisch2001galerkin,kunisch2002galerkin}) solves problem (\ref{minPOD}).
\begin{prop}
Let $\lambda_1 \geq \lambda_2 \geq \cdots \geq \lambda_l > 0$ denote the positive eigenvalues of $K$ and $\bm{v}_1, \bm{v}_2, \ldots,\bm{v}_l$ the associated orthonormal eigenvectors. Then a POD basis of rank $d\leq l$ is given by
\begin{eqnarray}
\psi_i = \frac{1}{\sqrt{\lambda_i}}\sum_{j=1}^\ell(\bm{v}_i)_jU_j, ~~~ i=1,\ldots,d,
\end{eqnarray}
where $(\bm{v}_i)_j$ denotes the $j$-th component of the eigenvector $\bm{v}_i$. Furthermore, the following error formula holds:
\begin{eqnarray}
 \frac{1}{\ell}\sum_{i=1}^\ell \lVert U_i - \sum_{j=1}^d(U_i,\psi_j)_{H_0^1(D)}\psi_j\rVert_{H_0^1(D)}^2 = \sum_{j=d+1}^l \lambda_j.
\end{eqnarray}
\end{prop}
Let $X^d := \text{span}\{\psi_1,\psi_2,\ldots,\psi_d\}$, and consider the problem of finding $w_d^m\in X^d \subset H_h(D)$ such that $w_d^0(\mathbf{x})=0$ and for $m=1,\ldots,N$,
\begin{eqnarray}
\label{eqn:DeterministicHeatPOD}
(w_d^m,v_d)+k(\nabla w_d^m,\nabla v_d) = k(g^m,v_d)+(w_d^{m-1},v_d),~~~\forall v_d\in X^d \subset H_h(D), ~~~m=1,\ldots,N.
\end{eqnarray}
\begin{rema}
If $\mathfrak{T}_h$ is a uniformly regular triangulation and $H_h(D)$ is the the space of piecewise linear functions, the total degrees of freedom for problem (\ref{eqn:DeterministicHeatFiniteElement}) is $N_h$, where $N_h$ is the number of vertices of triangles in $\mathfrak{T}_h$, while the total of degrees of freedom for problem (\ref{eqn:DeterministicHeatPOD}) is $d$ (where $d \ll l \ll \ell \ll N_h$).
\end{rema}
The following proposition, proved in \cite{luo2009finite}, gives us an error estimate on the solution of problem (\ref{eqn:DeterministicHeatPOD}).
\begin{prop}
\label{prop:SinglePOD}
If $w_h^m\in H_h(D)$ is the solution of problem (\ref{eqn:DeterministicHeatFiniteElement}), $w_d^m\in X^d \subset H_h(D)$ is the solution of problem (\ref{eqn:DeterministicHeatPOD}), $k = O(h)$, $\ell^2 = O(N)$, and snapshots are equably taken, then for $m = 1,2,\ldots,N$, the following estimates hold:
\begin{equation}
\label{PODErrorEstimates}
\lVert w_h^m-w_d^m\lVert_0 + \frac{1}{\ell}\sum_{j=1}^\ell \lVert\nabla(w_h^{m_j}-w_d^{m_j})\lVert_0 \leq C\left(k^{1/2}\sum_{j=d+1}^l\lambda_j\right)^{1/2}, ~~m = m_i, ~~i=1,\ldots,\ell;
\end{equation}
\begin{equation}
\lVert w_h^m-w_d^m\lVert_0 + \frac{1}{\ell}\left[\lVert \nabla (w_h^m-w_d^m)\lVert_0 + \sum_{j=1}^{\ell-1} \lVert\nabla(w_h^{m_j}-w_d^{m_j})\lVert_0\right] \leq C\left(k^{1/2}\sum_{j=d+1}^l\lambda_j\right)^{1/2}+Ck, ~m\neq m_i.\nonumber
\end{equation}
\end{prop}
Combining (\ref{FEErrorEstimates}) and (\ref{PODErrorEstimates}) we get the following result.
\begin{prop}
Under assumptions of proposition \ref{prop:SinglePOD}, the error estimate between the solutions of problems (\ref{eqn:DeterministicHeat}) and (\ref{eqn:DeterministicHeatPOD}), for $m = 1,2,\ldots,N$, is given by:
\begin{eqnarray}
\label{mainPODerrorEstimate}
\lVert w^m-w_d^m\lVert_0 \leq Ch^{s+1}+Ck+C\left(k^{1/2}\sum_{j=d+1}^l\lambda_j\right)^{1/2}
\end{eqnarray}
\end{prop}
Now, with a slight misusing of notation, we assume that the function $f$ is given by $f = f(\y)$, where $f\in C(\Gamma;C(0,T;L^2(D)))$ is the function employed in equation (\ref{eqn:weak3}), and consider the following problem: find $u\in H_0^1(D)$ such that $u(\x,0)=0$, for all $\x \in D$, and
\begin{eqnarray}
\label{eqn:DeterministicHeatPurturbed}
(u_t,v) + (\nabla u, \nabla v) = (f,v),~~~~\forall v \in H_0^1(D).
\end{eqnarray}
\begin{rema}
\label{LipschitzAssumtion}
Note that since $\lVert\y - \yp\lVert < \eta$ and under the assumption that $f\in C(\Gamma;C(0,T;L^2(D)))$ is Lipschitz continuous on $\Gamma$,  we get that $\lVert f(\y) - f(\yp)\lVert_{C(0,T;L^2(D))} = \lVert f - g\lVert_{C(0,T;L^2(D))}\leq L_f\lVert \y -\yp\lVert$, where $L_f$ is the Lipschitz constant. Also, note that we are slightly misusing the symbol $f$ to denote both the function $f\in C(\Gamma;C(0,T;L^2(D)))$ employed in equation (\ref{eqn:weak3}) and the function $f = f(\yp) \in C(0,T;L^2(D))$ used in equation (\ref{eqn:DeterministicHeat}).
\end{rema}
Let us also consider the following problem: find $u_d^m\in X^d \subset H_h(D)$ such that $u_d^0(\mathbf{x})=0$ and for $m=1,\ldots,N$,
\begin{eqnarray}
\label{eqn:DeterministicHeatPODPurturbed}
(u_d^m,v_d)+k(\nabla u_d^m,\nabla v_d) = k(f^m,v_d)+(u_d^{m-1},v_d),~~~\forall v_d\in X^d \subset H_h(D), ~~~m=1,\ldots,N.
\end{eqnarray}
Note that equations (\ref{eqn:DeterministicHeatPODPurturbed}) and (\ref{part1}) are identical, using the fact that we are using $f = f(\y)$. Our aim is to find an estimate for $\lVert u^m - u_d^m\Vert_0$. First we need to prove two lemmas.
\begin{lemm}
Let $u$ be the solution of problem (\ref{eqn:DeterministicHeatPurturbed}) and let $w$ be the solution of problem (\ref{eqn:DeterministicHeat}), then we have:
\begin{eqnarray}
\label{perturbationError}
\lVert u^m - w^m\lVert_0 \leq C\lVert f - g\lVert_{C(0,T;L^2(D))} .
\end{eqnarray}
\end{lemm}
\begin{proof}
let $z = u - w$ and subtract equations (\ref{eqn:DeterministicHeat}) and (\ref{eqn:DeterministicHeatPurturbed}) to get:
\begin{eqnarray}
\label{eqn:errorHeat}
(z_t,v) + (\nabla z, \nabla v) = (f-g,v),~~~~\forall v \in H_0^1(D),
\end{eqnarray}
with $z(\x,0)=0$, for all $\x\in D$. Letting $v = z$ and integrating equation (\ref{eqn:errorHeat}) from 0 to $t_m$, we get:
\[\frac{1}{2}\int_0^{t_m}\frac{d}{dt}\lVert z\lVert_0^2dt + \int_0^{t_m}(\nabla z,\nabla z)dt = \int_0^{t_m}(f-g,z)dt.\]
This results in
\[\frac{1}{2}\lVert z^m\lVert_0^2\leq \int_0^{t_m}\lVert f-g\lVert_0 \lVert z\lVert_0dt \leq \frac{1}{2}\int_0^T\lVert f-g\lVert_0^2dt +  \frac{1}{2}\int_0^T\lVert z\lVert_0^2dt.\]
Therefore,
\begin{eqnarray}
\label{eqn:intermediate}
\lVert z^m\lVert^2_0\leq T\lVert f -g \lVert_{C(0,T;L^2(D))}^2 + \int_0^T\lVert z\lVert_0^2dt.
\end{eqnarray}
Now we need to bound $ \int_0^T\lVert z\lVert_0^2dt$. For this, we integrate (\ref{eqn:errorHeat}) once again but this time upto $T$, and use the Poincar\'{e} inequality $\lVert v\lVert_0 \leq C_p\lVert\nabla v\lVert_0$, for each $v \in H_0^1(D)$, to get:
\[\frac{1}{2}\lVert z(T)\lVert_0^2 + \frac{1}{C_p^2}\int_0^T\lVert z\lVert_0^2dt \leq \int_0^T\lVert f-g\lVert_0 \lVert z\lVert_0dt.\]
Therefore,
\[\int_0^T\lVert z\lVert_0^2dt\leq C_p^2\left(\frac{1}{2\delta}\int_0^T\lVert f-g\lVert_0^2dt + \frac{\delta}{2}\int_0^T\lVert z\lVert_0^2dt\right).\]
Thus,
\[(1-\frac{C_p^2}{2}\delta)\int_0^T\lVert z\lVert_0^2dt\leq \frac{C_p^2}{2\delta}T\lVert f-g\lVert_{C(0,T;L^2(D))}^2 .\]
Choose $\delta > 0$ such that $1-\frac{C_p^2}{2}\delta>0$, and let
\[C = \sqrt{T\left(1+\frac{C_p^2}{2\delta-C_p^2\delta^2}\right)}.\]
Now, equation (\ref{eqn:intermediate}) implies (\ref{perturbationError}).
\end{proof}
\begin{lemm}
\label{lemma:perturbationErrorPOD}
Let $u_d^m$ be the solution of problem (\ref{eqn:DeterministicHeatPODPurturbed}) and $w_d^m$ be the solution of problem (\ref{eqn:DeterministicHeatPOD}), then we have:
\begin{eqnarray}
\label{perturbationErrorPOD}
\lVert u_d^m - w_d^m \lVert_0 \leq C\lVert f - g\lVert_{C(0,T;L^2(D))}
\end{eqnarray}
\end{lemm}
\begin{proof}
let $z_d^m = u_d^m - w_d^m$ and subtract equations (\ref{eqn:DeterministicHeatPOD}) and (\ref{eqn:DeterministicHeatPODPurturbed}) to get:
\begin{eqnarray}
\label{eqn:errorHeatPOD}
(z_d^m,v_d)+k(\nabla z_d^m,\nabla v_d) = k(f^m-g^m,v_d)+(z_d^{m-1},v_d),~~~\forall v_d\in X^d \subset H_h(D), ~~~m=1,\ldots,N,
\end{eqnarray}
with  $z_d^0(\mathbf{x})=0$. Let $v_d = z_d^m$ in equation (\ref{eqn:errorHeatPOD}) and use Poincar\'{e} inequality $\lVert v\lVert_0 \leq C_p\lVert\nabla v\lVert_0$, for each $v \in H_0^1(D)$, to achieve:
\[\lVert z_d^m\lVert_0^2 + k\frac{1}{C_p^2}\lVert z_d^m\lVert_0^2\leq k\lVert f^m-g^m\lVert_0\lVert z_d^m\lVert_0 + \lVert z_d^{m-1}\lVert_0\lVert z_d^m\lVert_0.\]
Therefore,
\[\left(1+ k\frac{1}{C_p^2}\right)\lVert z_d^m\lVert_0\leq k\lVert f^m-g^m\lVert_0 + \lVert z_d^{m-1}\lVert_0,\]
which upon summation yields,
\[\lVert z_d^m\lVert_0\leq k \lVert f-g\lVert_{C(0,T;L^2(D))}\sum_{j=1}^m\left(\frac{1}{1+\frac{k}{C_p^2}}\right)^j.\]
Let $\gamma = \frac{1}{C_p^2}$ and note that $(1 + \gamma k)^m\leq e^{\gamma k m}$. Moreover, setting $\zeta = 1/(1+\gamma k)$ we find:
\[k\sum_{j=1}^m\left(\frac{1}{1+\frac{k}{C_p^2}}\right)^j = k\frac{1-\zeta^m}{\zeta^{-1}-1} = \frac{1-\zeta^m}{\gamma}\leq \frac{1-e^{-\gamma k m} }{\gamma}.\]
Letting $C = (1-e^{-\gamma k m})/\gamma$, we get (\ref{perturbationErrorPOD}).
\end{proof}
Now using estimates (\ref{mainPODerrorEstimate}), (\ref{perturbationError}) and (\ref{perturbationErrorPOD}) and remark \ref{LipschitzAssumtion}, we get the following error estimate.
\begin{theo}
Let $u$ be the solution of problem (\ref{eqn:DeterministicHeatPurturbed}), and $u_d^m$ be the solution of problem (\ref{eqn:DeterministicHeatPODPurturbed}), for $m = 1,\ldots,N$, we have
\begin{eqnarray}
\label{ess:DeterministicL2Norm}
\lVert u^m - u_d^m\lVert_0\leq C \eta + Ch^{s+1}+Ck+C\left(k^{1/2}\sum_{j=d+1}^l\lambda_j\right)^{1/2},
\end{eqnarray}
where the eigenvalues $\lambda_j$ depend on $\yp\in B_\eta(\y)\subset \Gamma$, and the constants C depend on $\y$ and $\yp$, but are independent of $h$, $k$ and $\eta$.
\end{theo}
\section{Error analysis}
In this section, we carry out an error analysis for the multi-fidelity collocation method introduced in section \ref{sec:MFCollocationMethod} for problem (\ref{eqn:weak3}). In \cite{babuvska2007stochastic}, the authors showed that if the solution of (\ref{eqn:weak3}) is analytic with respect to the random parameters, then the collocation scheme (\ref{CollocationScheme}) attains an exponential error decay for $u_d^m - u_{d,\mathbf{p}}^m$ with respect to each $p_n$. The convergence proof in \cite{babuvska2007stochastic} applies directly to our case. Therefore, our main task is to prove the analyticity property of the POD solution $u_d^m$ eith respect to each random variable $y_n$. We will then only state the corresponding convergence result. In the following we impose similar restrictions on $f$ as in \cite{babuvska2007stochastic,zhang2012error}, \ie $f$ is continuous with respect to each element $\y \in \Gamma$ and that it has at most exponential growth at infinity, whenever the domain $\Gamma$ is unbounded. Moreover, we assume that joint density function $\rho$ behaves like a Gaussian kernel at infinity. In order to make it precise, we introduce the weight function $\bm{\sigma(\y)}=  \prod_{n=1}^r\sigma_n(y_n) \leq 1$, where
\[\sigma_n(y_n) = \left\{\begin{array}{ll}
1 & \text{if } \Gamma_n \text{ is bounded,} \\ 
e^{-\alpha_n\lvert y_n \lvert}\text{ for some $\alpha_n>0$} & \text{if } \Gamma_n \text{ is unbounded,}
\end{array}  \right.\]
and the space
\[C_{\bm{\sigma}}^0(\Gamma;V)=\{v:\Gamma \rightarrow V: v \text{ is continuous in $\y$ and $\text{max}_{\y\in\Gamma}\{\bm{\sigma}(\y)\lVert v(\y)\lVert_V\}<+\infty$}\},\]
where $V$ is a Banach space. In what follows, we assume that $f\in C_{\bm{\sigma}}^0(\Gamma;C([0,T];L^2(D)))$ and the joint probability density $\rho$ satisfies
\begin{eqnarray}
\label{ass:onRho}
\rho(\y)\leq C_Me^{-\sum_{n=1}^r(\delta_n y_n)^2}, ~~~\forall \y \in \Gamma,
\end{eqnarray}
for some constant $C_M>0$, with $\delta_n$ being strictly positive if $\Gamma_n$ is unbounded and zero otherwise. Under these assumptions, the following proposition is immediate; see \cite{babuvska2007stochastic}.
\begin{prop}
The solution of problem (\ref{eqn:weak3}) satisfies $u\in C_{\bm{\sigma}}^0(\Gamma;C(0,T;H_0^1(D)))$ and correspondingly, the approximate solution $u_d^m$ resulted from (\ref{eqn:DeterministicHeatPODPurturbed}) or equivalently (\ref{part1}), satisfies $u_d^m\in C_{\bm{\sigma}}^0(\Gamma;H_h(D))$, for $m=1,\ldots,N$.
\end{prop}
Furthermore, we have the following regularity result.
\begin{lemm}
The following energy estimate holds:
\[\lVert u_{d}^m\lVert_{L^2(D)\otimes L_\rho^2(\Gamma)}\leq C_p^2(1-e^{-\frac{km}{C_p^2}})\lVert f\lVert_{C(0,T;L^2(D))\otimes L_\rho^2(\Gamma)},\]
where $C_p$ is the Poincar\'{e} canstant.
\end{lemm}
\begin{proof}
Similar to the proof of lemma \ref{lemma:perturbationErrorPOD}.
\end{proof}
\subsection{Analyticity with respect to random parameters}
We prove that the solution $u_d^m$ of equation (\ref{eqn:DeterministicHeatPODPurturbed}) is analytic with respect to each random parameter $y_n\in\Gamma$, whenever $f(\y)$ is infinitely differentiable with respect to each component of $\y$. To do this, we introduce the following notations as in \cite{babuvska2007stochastic,zhang2012error}:
\[\y_n^* \in \Gamma_n^* =\prod_{j=1,j\neq n}^r\Gamma_j ~~~~~\text{and}~~~~~~\bm{\sigma}_n^*=\prod_{j=1,j\neq n}^r \sigma_j.\]
We first make the additional assumption that for every $\y = (y_n,\y_n^*)\in \Gamma$, there exists $\gamma_n < +\infty$ such that
\begin{eqnarray}
\label{ass:OnF}
\frac{\lVert \partial_{y_n}^jf(\y) \lVert_{C(0,T;L^2(D))}}{1 + \lVert f(\y) \lVert_{C(0,T;L^2(D))}}\leq \gamma_n^j j! .
\end{eqnarray}
\begin{rema}
Under the finite dimensional noise assumption (\ref{ass:FiniteNoise}), $f(t,\x,\omega)$ is represented by a truncated linear or nonlinear expansion so that assumption (\ref{ass:OnF}) holds. For example, consider a truncated KL expansion for random forcing term $f(t,\x,\omega)$ given by
\begin{equation}
\label{fKLExpansion}
f(t,\x,\omega) = f(t,\x,\y(\omega)) = \mathbb{E}[f](t,\x) + \sum_{n=1}^r\sqrt{\mu_n}c_n(t,\x)y_n(\omega).
\end{equation}
We have
\[\frac{\lVert \partial_{y_n}^jf(\y) \lVert_{C(0,T;L^2(D))}}{1 + \lVert f(\y) \lVert_{C(0,T;L^2(D))}}\leq \left\{\begin{array}{ll}
\sqrt{\mu_n}\lVert c_n \lVert_{C(0,T;L^2(D))},& j=1, \\ 
0, & j>1.
\end{array}  \right.\]
Therefore, we can set $\gamma_n = \sqrt{\mu_n}\lVert c_n \lVert_{C(0,T;L^2(D))}$, and observe that definition (\ref{fKLExpansion}) satisfies assumption (\ref{ass:OnF}). Moreover, the random forcing $f(t,\x,\y)$ defined in (\ref{fKLExpansion}), satisfies the Lipschitz continuity assumption of remark \ref{LipschitzAssumtion}.
\end{rema}
\begin{lemm}
Under assumption (\ref{ass:OnF}), if the solution $u_d^m(\x,y_n,\y_n^*)$ is considered as a function of $y_n$, \ie $u_d^m:\Gamma_n \rightarrow C_{\bm{\sigma}_n^*}^0(\Gamma_n^*;L^2(D))$, then the $j$-th derivative of $u_d^m(\x,\y)$ with respect to $y_n$ satisfies
\begin{eqnarray}
\label{ess:Analiticity}
\lVert \partial_{y_n}^ju_d^m(\y)\lVert_{L^2(D)}\leq Cj!\gamma_n^j, ~~~m=1,\ldots,N,
\end{eqnarray}
where $C$ depends on $\lVert f(\y)\lVert_{C(0,T;L^2(D))}$, and the Poincar\'{e} constant $C_p$.
\end{lemm}
\begin{proof}
Take the $j$-th derivative of formulation (\ref{eqn:DeterministicHeatPODPurturbed}) or equivalently (\ref{part1}) with respect to $y_n$, and let $v_d = \partial_{y_n}^ju_d^m(\y)$ to get
\[\lVert\partial_{y_n}^ju_d^m(\y)\lVert_0^2+k\lVert\partial_{y_n}^j\nabla u_d^m(\y)\lVert_0^2 = k(\partial_{y_n}^jf^m(\y),\partial_{y_n}^ju_d^m(\y))+(\partial_{y_n}^ju_d^{m-1}(\y),\partial_{y_n}^ju_d^m(\y)).\]
Therefore,
\[(1+\frac{k}{C_p^2})\lVert\partial_{y_n}^ju_d^m(\y)\lVert_0\leq k\lVert \partial_{y_n}^jf^m(\y)\lVert_0 + \lVert\partial_{y_n}^ju_d^{m-1}(\y)\lVert_0,\]
which upon summation yields
\[\lVert\partial_{y_n}^ju_d^m(\y)\lVert_0\leq k \lVert \partial_{y_n}^jf(\y)\lVert_{C(0,T;L^2(D))}\sum_{i=1}^m\left(\frac{1}{1+\frac{k}{C_p^2}}\right)^i.\]
Thus,
\[\lVert\partial_{y_n}^ju_d^m(\y)\lVert_0 \leq C_p^2(1-e^{-\frac{km}{C_p^2}})[1 + \lVert f(\y) \lVert_{C(0,T;L^2(D))}] \gamma_n^j j!.\]
Letting $C = C_p^2(1-e^{-\frac{km}{C_p^2}})[1 + \lVert f(\y) \lVert_{C(0,T;L^2(D))}]$ we get (\ref{ess:Analiticity}).
\end{proof}
We will immediately obtain the following theorem, whose proof closely follows the proof of theorem 4.4 in \cite{zhang2012error}.
\begin{theo}
\label{thm:AnalyticExtension}
Under assumption (\ref{ass:OnF}), the solution $u_d^m(\x,y_n,\y_n^*)$ considered as a function of $y_n$, admits an analytic extension $u_d^m(\x,z,\y_n^*)$, $z\in\mathbb{C}$, in the region of complex plane
\[\Sigma(\Gamma_n,\tau_n) := \{z\in\mathbb{C}:\text{dist}(z,\Gamma_n) \leq \tau_n\},\]
where $0<\tau_n<1/\gamma_n$.
\end{theo}
\begin{proof}
For each $y_n\in\Gamma_n$ we define the power series $u_d^m:\mathbb{C}\rightarrow C_{\bm{\sigma}_n^*}^0(\Gamma_n^*;L^2(D))$ as
\[u_d^m(\x,z,\y_n^*)=\sum_{j=0}^\infty\frac{(z-y_n)^j}{j!}\partial_{y_n}^ju_d^m(\x,y_n,\y_n^*)\]
Thus,
\begin{eqnarray}
\sigma_n(y_n)\lVert u_d^m(z)\lVert_{C_{\bm{\sigma}_n^*}^0(\Gamma_n^*;L^2(D))}&\leq &\sum_{j=0}^\infty\frac{\lvert z-y_n\lvert^j}{j!}\lVert \partial_{y_n}^ju_d^m(y_n)\lVert_{C_{\bm{\sigma}_n^*}^0(\Gamma_n^*;L^2(D))} \nonumber \\
&\leq &\sigma_n(y_n)C(y_n)\sum_{j=0}^\infty(\lvert z-y_n\lvert\gamma_n)^j\nonumber \leq\hat{C} \sum_{j=0}^\infty(\lvert z-y_n\lvert\gamma_n)^j,
\end{eqnarray}
where $C(y_n)$ is a function of $\lVert f(y_n)\lVert_{C_{\bm{\sigma}_n^*}^0(\Gamma_n^*;C(0,T;L^2(D)))}$, and the constant $\hat{C}$ is a function of $\lVert f \lVert_{C_{\bm{\sigma}}^0(\Gamma;C(0,T;L^2(D)))}$. The series is convergent for all $z\in\mathbb{C}$, provided that $\lvert z - y_n\lvert \leq \tau_n < 1/\gamma_n$. Therefore, the function $u_d^m$ admits an analytic extension in the region $\Sigma(\Gamma_n;\tau_n)$.
\end{proof}
\subsection{Convergence analysis}
Our goal is to provide an estimate for the total error $e^m = u^m - u_{d,\mathbf{p}}^m$ in the norm $L^2(D)\otimes L_\rho^2(\Gamma)$, for each $m=1,\ldots,N$. The error splits naturally into $e^m =(u^m - u_{d}^m) + (u_{d}^m -u_{d,\mathbf{p}}^m)$. Recall that $u_{d,\mathbf{p}}^m = \mathcal{I}_{\mathbf{p}}u_d^m$ and is given by (\ref{CollocationScheme}). We can estimate the interpolation error $(u_{d}^m -u_{d,\mathbf{p}}^m)$ by repeating the same procedure as in \cite{babuvska2007stochastic}, using the analyticity result of theorem \ref{thm:AnalyticExtension}. All details about the estimates of the interpolation error can be found in section 4 of  \cite{babuvska2007stochastic} and the references cited therein. Therefore we state the following theorem without proof.
\begin{theo}
\label{thm:StochasticPart}
Under assumption (\ref{ass:OnF}), there exist positive constants $b_n, n=1,\ldots,r,$ and $C$ that are independent of $h, d$, and $\mathbf{p}$ such that
\begin{eqnarray}
\label{ess:StochasticPart}
\vn{u_d^m - u_{d,\mathbf{p}}^m}_{L^2(D)\otimes L_\rho ^2(\Gamma)} \leq C\sum_{n=1}^r\beta_n(p_n)\exp({-b_np_n^{\theta_n}}),
\end{eqnarray}
where
\[\theta_n = \beta_n = 1 \text{~~~and~~~} b_n = \log\left[ \frac{2\tau_n}{\lvert \Gamma_n\lvert}\left(1 + \sqrt{1 + \frac{\lvert \Gamma_n\lvert^2}{4\tau_n^2}} \right)\right]~~~~~\text{if ~$\Gamma_n$ is bounded},\]
and
\[\theta_n = \frac{1}{2},~~~\beta_n = O(\sqrt{p_n}), \text{~~~and~~~ }b_n = \tau_n\delta_n~~~~~\text{if ~$\Gamma_n$ is unbounded},\]
where $\tau_n$ is the minimum distance between $\Gamma_n$ and the nearest singularity in the complex plane, as defined in theorem~\ref{thm:AnalyticExtension}, and $\delta_n$ is defined in assumption (\ref{ass:onRho}).
\end{theo}
\begin{rema}[Convergence with respect to the number of collocation points] For an isotropic full tensor-product approximation, \ie $p_1=p_2=\cdots=p_r=p$, the number of collocation points $\Theta$ is given by $\Theta=(1+p)^d$. Thus, one can easily obtain the following error bound with respect to $\Theta$; see \cite{zhang2012error}.
\begin{eqnarray}
\label{ess:StochasticPartNumberOfPoints}
\vn{u_d^m - u_{d,\mathbf{p}}^m}_{L^2(D)\otimes L_\rho ^2(\Gamma)} \leq \left\{\begin{array}{ll}
C\Theta^{-b_{\text{min}}/{r}}, & \text{if $\Gamma$ is bounded,} \\ 
C\Theta^{-b_{\text{min}}/{2r}}, & \text{if $\Gamma$ is unbounded,}
\end{array}  \right.
\end{eqnarray}
where $b_{\text{min}} = \text{min}\{b_1,b_2,\ldots,b_r\}$ as in theorem \ref{thm:StochasticPart}. The constant $C$ depends on $r$ and $b_\text{min}$.
\end{rema}
\begin{rema}[Extensions to sparse grid stochastic collocation methods]
\label{rem:SparseGrid}
Note that the convergence as shown in (\ref{ess:StochasticPartNumberOfPoints}) becomes slower as the dimension $r$ increases. This slow-down effect as a result of increase in dimension is called the \textit{curse of dimensionality}. For large values of $r$, sparse grid stochastic collocation methods \cite{nobile2009analysis,nobile2008sparse}, specially adaptive and anisotropic ones, \eg \cite{ma2009adaptive,nobile2008anisotropic} are more effective in dealing with this problem. Our analyticity result (theorem \ref{thm:AnalyticExtension}) combined with the analysis in \cite{nobile2009analysis,nobile2008sparse,ma2009adaptive,nobile2008anisotropic}, can easily lead to the derivation of error bounds for sparse grid approximations. For instance, for an isotropic Smolyak approximation \cite{nobile2009analysis,nobile2008sparse} with a total of $\Theta$ sparse grid points, the error can be bounded by \[C\Theta^{-b_\text{min}/(1+\log(2r))}.\] Here, we will give a short description of the isotropic Smolyak algorithm. More detailed information can be found in \cite{barthelmann2000high,nobile2008sparse}. Assume $p_1=p_2=\cdots=p_r=p$. For $r=1$, let $\{\mathcal{I}_{1,i}\}_{i=1,2,\ldots}$ be a sequence of interpolation operators given by equation (\ref{CollocationScheme}). Define $\Delta_0 = \mathcal{I}_{1,0} = 0$ and $\Delta_i = \mathcal{I}_{1,i} - \mathcal{I}_{1,i-1}$. Now for $r > 1$, let
\begin{eqnarray}
\mathcal{A}(q,r) = \sum_{0\leq i_1+i_2+\ldots+i_r \leq q}\Delta_{i_1}\otimes\cdots\otimes \Delta_{i_r}
\end{eqnarray}
where $q$ is a non-negative integer. $\mathcal{A}(q,r)$ is the Smolyak operator, and $q$ is known as the sparse grid \textit{level}.
\end{rema}
Now we need to find error bounds for the deterministic part of our algorithm in the $L^2(D)\otimes L_\rho^2(\Gamma)$ norm, \ie $u^m - u_d^m$. First, note that according to (\ref{ass:onRho}), the joint density function $\rho$ behaves like a Gaussian kernel at infinity. Therefore, in practice we are literally dealing with a compact random parameter set $\Gamma$, since we can approximate $\Gamma$ with a large enough compact set. So from now on we assume that $\Gamma$ is compact. We know that $\Gamma \subset \bigcup_{\yp\in\Gamma}B_\eta(\yp)$. Thus, using the compactness assumption on $\Gamma$, there exist $\Upsilon\in\mathbb{N}_+$ and $\{{^i}\yp\}_{i=1}^\Upsilon\subset\Gamma$ such that $\Gamma = \bigcup_{i=1}^\Upsilon B_\eta({^i}\yp)\cap\Gamma$. Letting ${^i}\Gamma = B_\eta({^i}\yp)\cap\Gamma$, we can write $\Gamma = \bigcup_{i=1}^\Upsilon {^i}\Gamma$.
\begin{theo}
Under the Lipschitz continuity (see remark \ref{LipschitzAssumtion}) assumption, the exist constants $C$ and $\Lambda$ such that
\begin{eqnarray}
\label{ess:DeterministicPart}
\vn{u^m - u_d^m}_{L^2(D)\otimes L_\rho ^2(\Gamma)}\leq C\eta + Ch^{s+1}+Ck+Ck^{1/4}\Lambda.
\end{eqnarray}
\end{theo}
\begin{proof}
Let us first integrate the the last term in estimate (\ref{ess:DeterministicL2Norm}). Thus, we have
\[\bigintss_\Gamma \left\{C(\y,\yp(\y))\left(k^{1/2}\sum_{j=d(\y,\yp(\y))+1}^{l(\y,\yp(\y))}\lambda_j(\yp(\y))\right)^{1/2}\right\}^2\rho(\y)d\y = \]
\[k^{1/2}\int_\Gamma C(\y,\yp(\y))^2\sum_{j=d(\y,\yp(\y))+1}^{l(\y,\yp(\y))}\lambda_j(\yp(\y))\rho(\y)d\y = \]
\[k^{1/2}\sum_{i=1}^\Upsilon\left(\sum_{j=d({^i}\yp)+1}^{l({^i}\yp)}\lambda_j({^i}\yp)\right)\int_{{^i}\Gamma}C(\y,{^i}\yp)^2\rho(\y)d\y\]
Now letting $\Lambda_i = \sum_{j=d({^i}\yp)+1}^{l({^i}\yp)}\lambda_j({^i}\yp)$, and assuming $\Lambda^2 = \max_{i=1,\ldots,\Upsilon}\{\Lambda_i\}$, we get the following upper bound for the above expression:
\[k^{1/2}\Lambda^2\sum_{i=1}^\Upsilon\int_{{^i}\Gamma}C(\y,{^i}\yp)^2\rho(\y)d\y = k^{1/2}\Lambda^2\int_\Gamma C(\y,\yp(\y))^2\rho(\y)d\y.\]
Letting $C^2  =\int_\Gamma C(\y,\yp(\y))^2\rho(\y)d\y$, we get the last term in (\ref{ess:DeterministicPart}). The first three terms of (\ref{ess:DeterministicPart}) can also be easily computed by integrating the first three terms of (\ref{ess:DeterministicL2Norm}). We will get the same expressions for the constants $C$ as above.
\end{proof}
\begin{rema}
Note that due to the way that the POD method works, the constant $\Lambda$ is so small that the $k^{1/4}$ term has a very little effect on the error.
\end{rema}
Combining (\ref{ess:StochasticPart}) and (\ref{ess:DeterministicPart}), we will finally get the following total error estimate.
\begin{theo}
Under assumption (\ref{ass:OnF}) and the Lipschitz continuity (see remark \ref{LipschitzAssumtion}) assumption , there exist positive constants $C$ and $\Lambda$ that are independent of $h, d, k, \eta$ and $\mathbf{p}$, and there exist constants $b_n, n=1,\ldots,r,$ such that
\begin{eqnarray}
\vn{u^m - u_{d,\mathbf{p}}^m}_{L^2(D)\otimes L_\rho ^2(\Gamma)} \leq C\eta + Ch^{s+1}+Ck+Ck^{1/4}\Lambda + C\sum_{n=1}^r\beta_n(p_n)\exp({-b_np_n^{\theta_n}}),
\end{eqnarray}
where $\theta_n, \beta_n$ and $p_n$ are the same as the ones in theorem \ref{thm:StochasticPart}.
\end{theo}
\begin{rema}
In some cases, one might be interested in estimating the expectation error, \ie $\vn{\mathbb{E}[u^m-u_{d,\mathbf{p}}^m]}_{L^2(D)}$. This can be easily achieved by observing that:
\begin{eqnarray}
\vn{\mathbb{E}[u^m-u_{d,\mathbf{p}}^m]}_{L^2(D)}^2 &=& \int_D\left[\int_\Gamma [u^m(\x,\y)-u_{d,\mathbf{p}}^m(\x,\y)]\rho(\y)d\y \right]^2d\x\nonumber\\
&\leq & \int_D\left[\int_\Gamma [u^m(\x,\y)-u_{d,\mathbf{p}}^m(\x,\y)]^2\rho(\y)d\y \int_\Gamma \rho(\y)d\y \right]d\x\nonumber\\
&= & \int_\Gamma\left[ \int_D [u^m(\x,\y)-u_{d,\mathbf{p}}^m(\x,\y)]^2d\x\right]\rho(\y)d\y\nonumber\\
& = &\vn{u^m - u_{d,\mathbf{p}}^m}_{L^2(D)\otimes L_\rho ^2(\Gamma)}.
\end{eqnarray}
\end{rema}
\section{Numerical experiments}
In this section, we provide a computational example to illustrate the advantages of multi-fidelity stochastic collocation method. Specifically, we consider problem (\ref{eqn:SParabolic}) with $D = (0,1)^2\subset \mathbb{R}^2$, $T = 1$, and the forcing term being given by:
\[f(t,\x,\omega) = 10 + e^t\sum_{n=1}^r\y_n(\omega)\sin (n\pi x).\]
The real-valued random variables $y_n, n=1,\ldots,r$, are supposed to be independent and have uniform distributions $U(0,1)$. In the followings, we let $r=4$. We employ the sparse grid stochastic collocation method introduced in remark \ref{rem:SparseGrid} with sparse grid level $q = 8$. We use the Clenshaw-Curtis abscissas (see \cite{clenshaw1960method}) as collocation points. These abscissas are the extrema of Chebyshev polynomials. We divide the spatial domain $D$ into $32\times 32$ small squares with side length $\Delta x = \Delta y = 1/32$, and then we connect the diagonals of the squares to divide each square into two triangles. These triangles consist the triangulation $\mathfrak{T}_h$, with $h = \sqrt{2}/32$. Take $k=0.1$ as the time step increment. We use all of the time steps to form the snapshots. We employ $6$ POD basis functions. In the following, we compare the solution resulting from a regular isotropic sparse grid stochastic collocation method which only uses the finite element method, with the hybrid multi-fidelity method proposed in this paper which employs both finite element and POD methods. In figure \ref{fig:Fig1}, we compare the expected values resulting from the multi-fidelity method and a regular sparse grid stochastic collocation method. We take $\eta = 0.1$. Recall that for each $\y \in \Gamma$ our method searches the $\eta$ neighbourhood of $\y$ to check whether for some $\yp \in B_\eta(\y)$ problem (\ref{eqn:weak3}) is already solved. If a nearby problem (at~$\yp$) is found to be solved by finite element method, our algorithm uses this information to create POD basis functions and solves problem (\ref{eqn:weak3}) at $\y$ using Galerkin-POD method which is computationally much cheaper than finite element. Moreover, figure \ref{fig:Fig2} compares variances of solutions resulting from the two methods. 

\begin{figure}[ht]
\centering
\begin{minipage}{.5\textwidth}
  \centering
  \psfig{figure=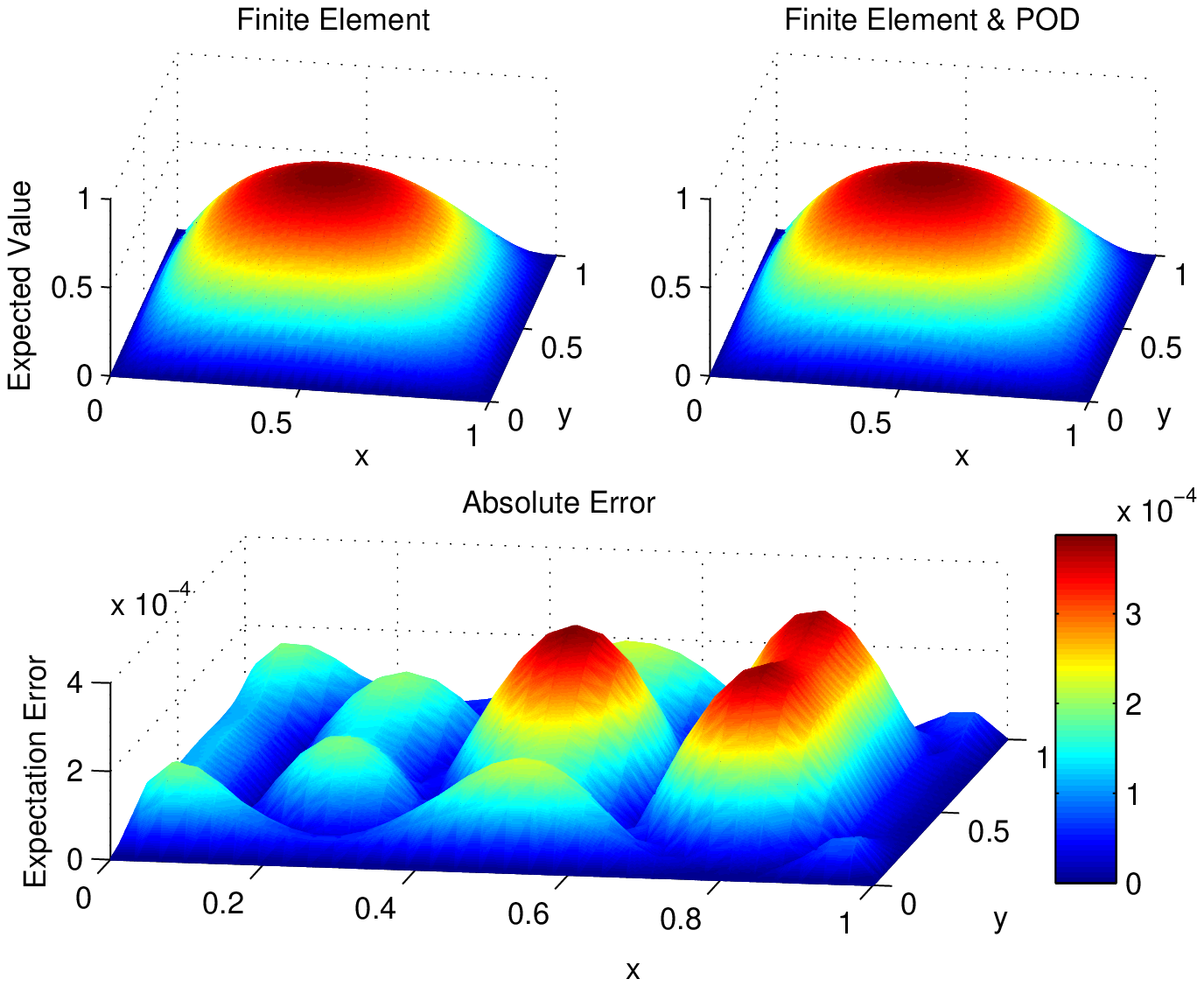,height=2.6in}
  \captionof{figure}{Comparison of expected values (bottom)\\ resulting from a regular sparse grid method (top left)\\ and the multi-fidelity method with $\eta = 0.1$ (top right).}
  \label{fig:Fig1}
\end{minipage}%
\begin{minipage}{.5\textwidth}
  \centering
  \psfig{figure=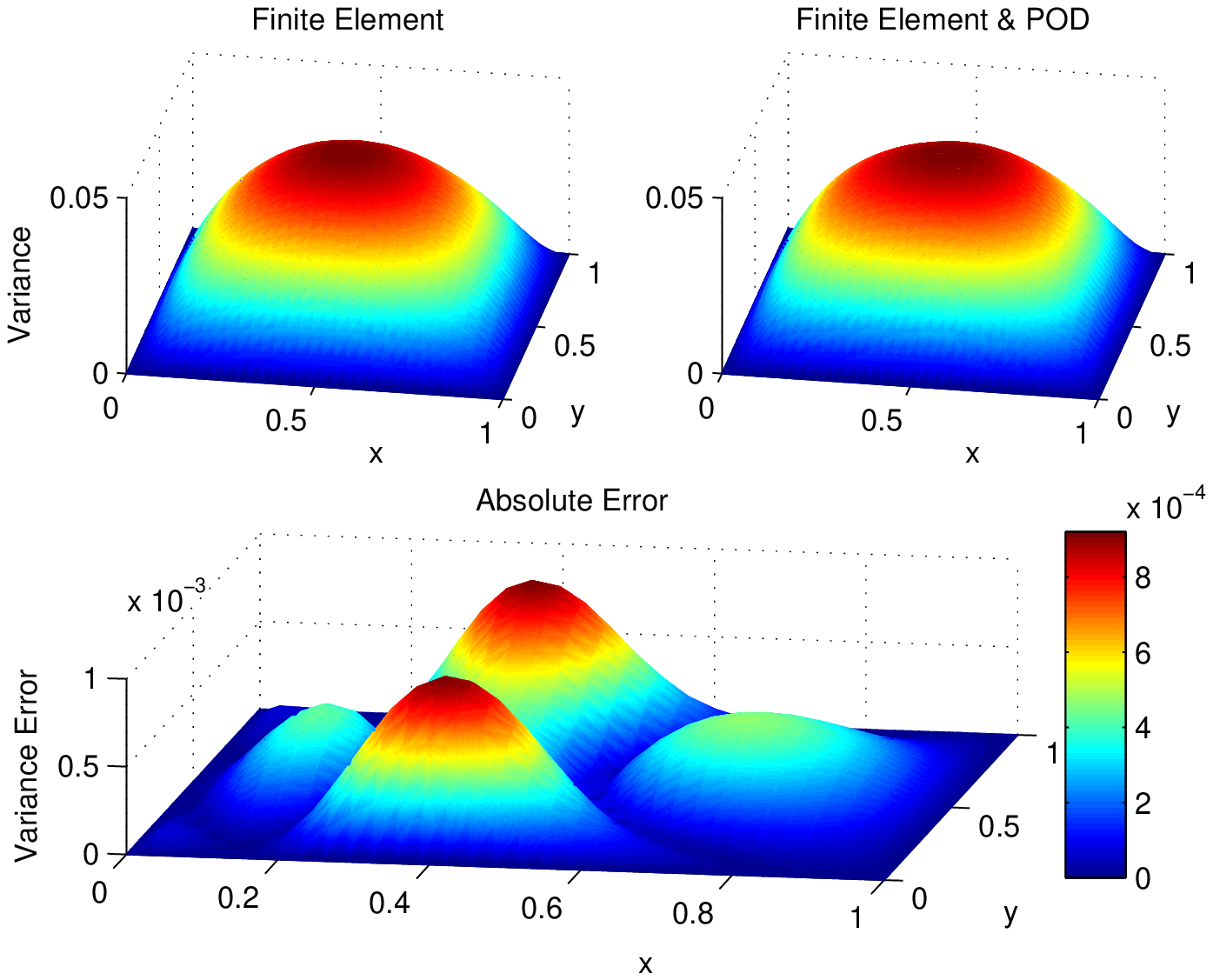,height=2.6in}
  \captionof{figure}{Comparison of variances of solutions (bottom)\\ resulting from a regular sparse grid method (top left)\\ and the multi-fidelity method with $\eta = 0.1$ (top right).}
  \label{fig:Fig2}
\end{minipage}
\end{figure}

Figures \ref{fig:Fig3} and \ref{fig:Fig4}, show the convergence patterns of expectations and variances of solutions with regard to $\eta$, respectively. These results validate our theoretical estimates of previous sections. We are actually comparing our multi-fidelity method with a regular sparse grid stochastic method. Note that for small enough $\eta$ (less than the shortest distance between the collocation points) we get the regular sparse grid method back. Therefore the error is zero for such a small $\eta$.\\

\begin{figure}[ht]
\centering
\begin{minipage}{.5\textwidth}
  \centering
  \psfig{figure=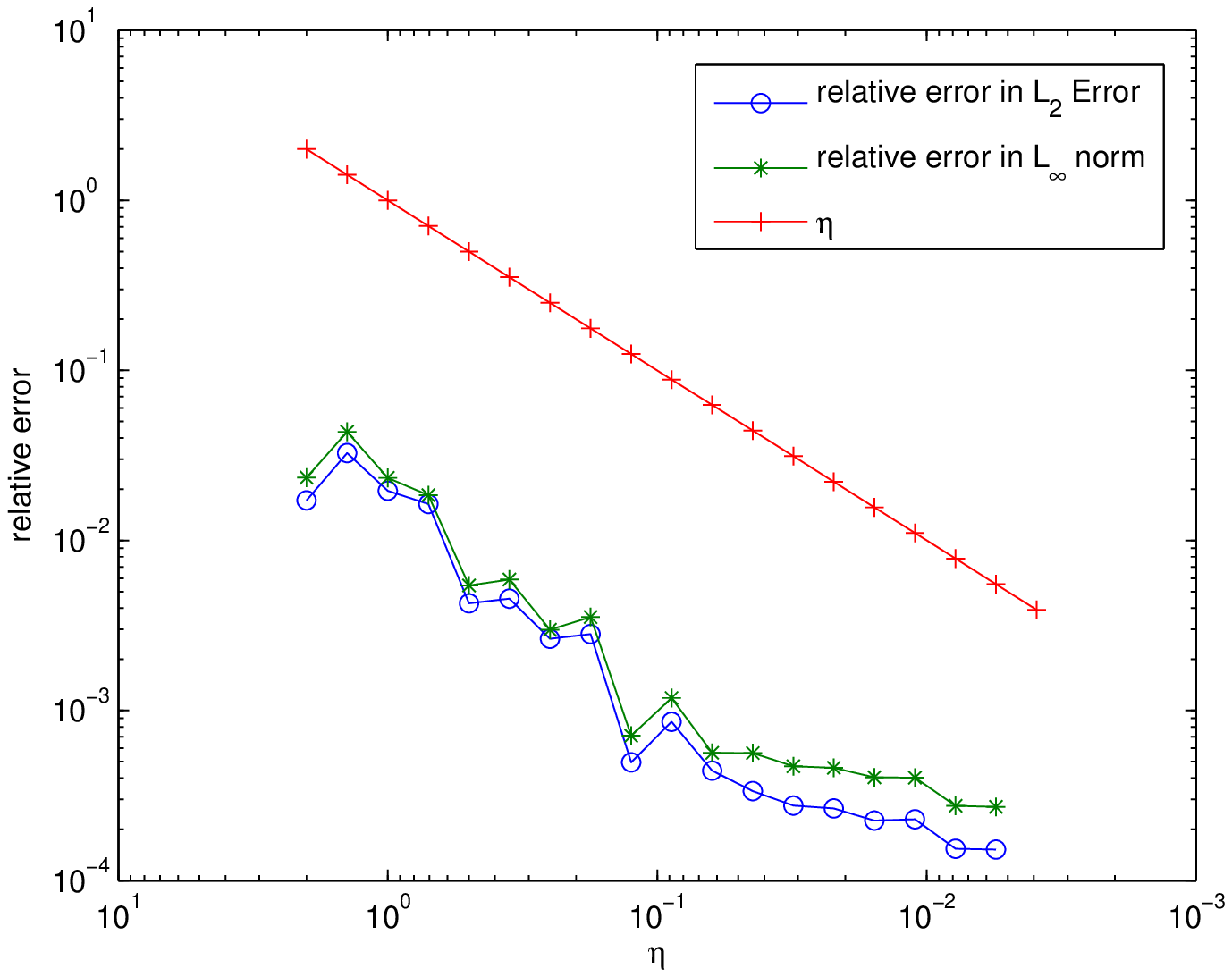,height=2.65in}
  \captionof{figure}{Convergence pattern of expected values of \\solutions with respect to $\eta$.}
  \label{fig:Fig3}
\end{minipage}%
\begin{minipage}{.5\textwidth}
  \centering
  \psfig{figure=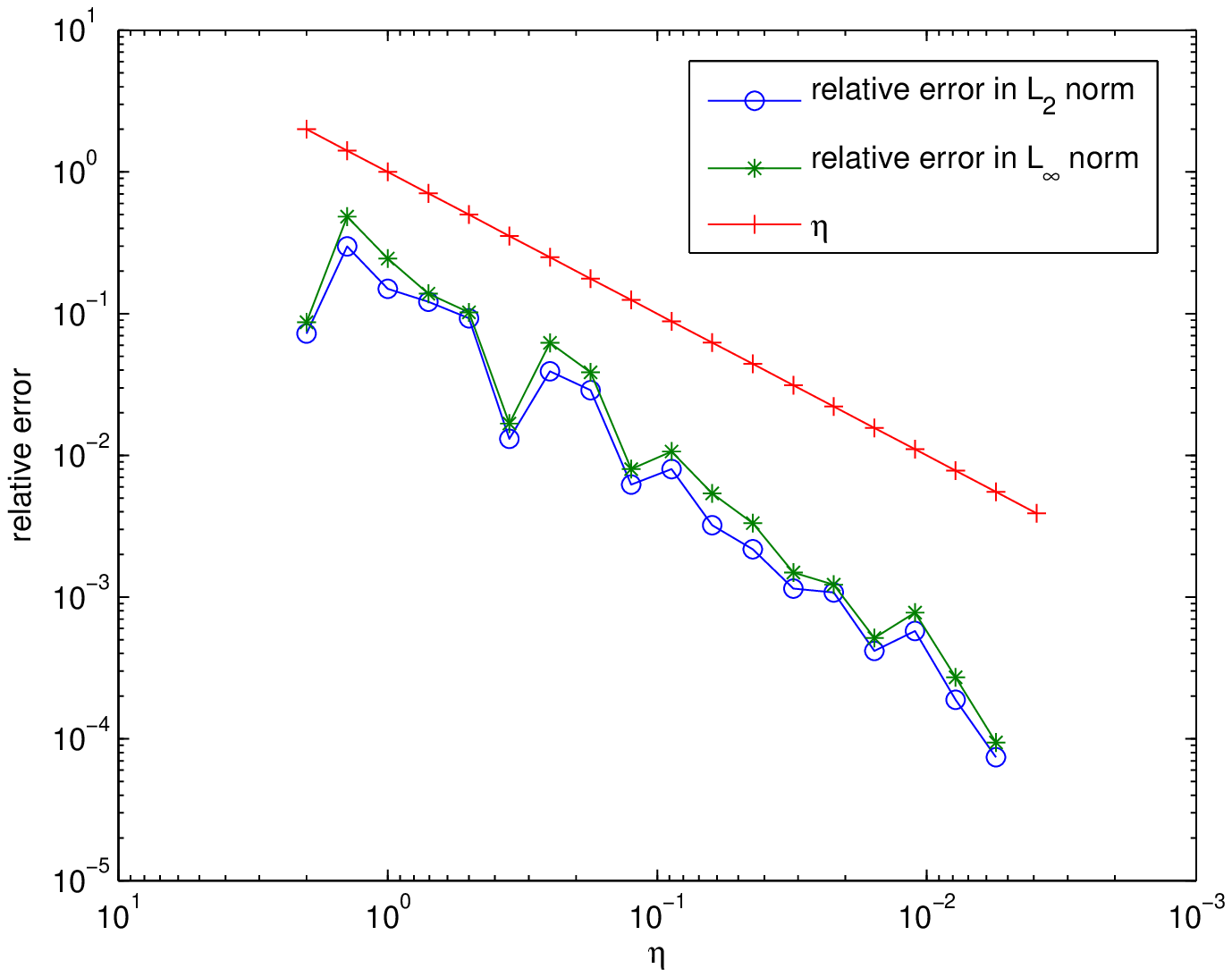,height=2.65in}
  \captionof{figure}{Convergence pattern of variances of solutions\\ with respect to $\eta$.}
  \label{fig:Fig4}
\end{minipage}
\end{figure}

\begin{figure}[h!]
\centering{ \psfig{figure=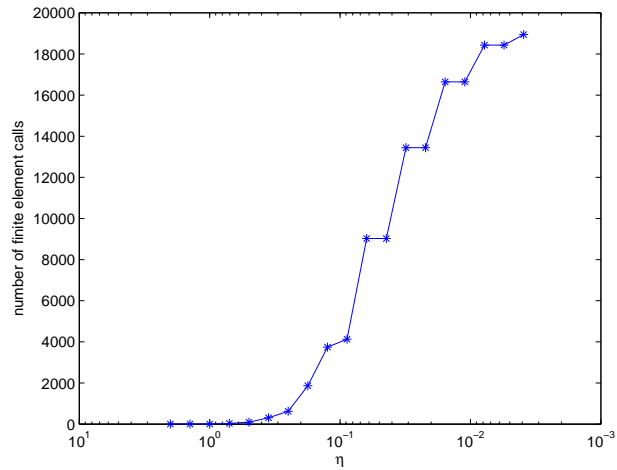,height=2.65in}}
\caption{The number of times that the finite element code is employed as a function of $\eta$.}
\label{fig:Fig5}
\end{figure}

Figure \ref{fig:Fig5} demonstrates how the number of times that the finite element code is employed increases with respect to a decrease in $\eta$.\\

Table \ref{tab:RelativeErrors}, summarizes the results when $\eta = 0.1$. In this case, the number of times that the finite element code is utilized by the multi-fidelity method is $3745$. Compared it to $18946$, the number of times that a regular sparse grid calls the finite element code.\\

\begin{table}[h!]
\begin{center}\caption{Relative errors when $\eta = 0.1$.}
\label{tab:RelativeErrors}
\begin{tabular}{|c|c|c|c|}
\hline 
• & Relative error in $L_2$ norm & Relative error in $L_\infty$ norm \\ 
\hline 
Expected value  & $3.6\times 10^{-4}$ & $4.8\times 10^{-4}$ \\ 
\hline 
Variance & $1.2\times 10^{-2}$ & $2.0\times 10^{-2}$ \\ 
\hline 
\end{tabular} 
\end{center}
\end{table}

Table \ref{tab:Summary} is just another way of presenting the data depicted in figures \ref{fig:Fig3},\ref{fig:Fig4}, and \ref{fig:Fig5}.

\begin{table}[h!]
\begin{center}\caption{Relative errors and the number of times that the finite element code is employed for different values of $\eta$.}
\label{tab:Summary}
\begin{tabular}{|c|c|c|c|c|c|}
\hline 
$\eta$ & $\#$ FE calls & Expectation $L_2$ error & Expectation $L_\infty$ error & Variance $L_2$ error & Variance $L_\infty$ error \\ 
\hline 
4 &	1 &	1.72E-02 &	2.34E-02 &	7.25E-02 &	8.72E-02\\
\hline
2 &	3 &	3.27E-02 &	4.35E-02 &	2.99E-01 &	4.84E-01\\
\hline
1 &	5 &	1.95E-02 &	2.33E-02 &	1.50E-01 &	2.45E-01\\
\hline
$1/2$ &	36 &	1.63E-02 &	1.85E-02 &	1.21E-01 &	1.38E-01\\
\hline
$(1/2)^2$ &	92	 & 4.26E-03 &	5.43E-03 &	9.27E-02 &	1.02E-01\\
\hline
$(1/2)^3$ &	306 &	4.55E-03 &	5.89E-03 &	1.31E-02 &	1.68E-02\\
\hline
$(1/2)^4$ &	621 &	2.64E-03 &	2.98E-03 &	3.91E-02 &	6.21E-02\\
\hline
$(1/2)^5$ &	1866	 & 2.81E-03 &	3.55E-03 &	2.88E-02 &	3.86E-02\\
\hline
$(1/2)^6$ &	3743 &	4.96E-04 &	7.09E-04 &	6.23E-03 &	8.00E-03\\
\hline
$(1/2)^7$ &	4129	 & 8.58E-04 &	1.19E-03 &	8.00E-03 &	1.07E-02\\
\hline
$(1/2)^8$ &	9026	 & 4.42E-04 &	5.63E-04 &	3.22E-03 &	5.39E-03\\
\hline
$(1/2)^9$ &	9026	 & 3.35E-04 &	5.61E-04 &	2.18E-03 &	3.33E-03\\
\hline
$(1/2)^{10}$ &	13442 &	2.76E-04 &	4.69E-04 &	1.15E-03 &	1.49E-03\\
\hline
$(1/2)^{11}$	 & 13442 &	2.65E-04 &	4.59E-04 &	1.08E-03 &	1.22E-03\\
\hline
$(1/2)^{12}$	 & 16642 &	2.25E-04 &	4.04E-04 &	4.18E-04 &	5.15E-04\\
\hline
$(1/2)^{13}$	 & 16642 &	2.29E-04 &	4.02E-04 &	5.75E-04 &	7.78E-04\\
\hline
$(1/2)^{14}$	 & 18434 &	1.54E-04 &	2.75E-04 &	1.89E-04 &	2.71E-04\\
\hline
$(1/2)^{15}$ &	18434 &	1.52E-04 &	2.71E-04 &	7.42E-05 &	9.43E-05\\
\hline
$(1/2)^{16}$ &	18946 &	0 &	0 &	0 &	0\\
\hline
\end{tabular} 
\end{center}
\end{table}
\begin{rema}
The method proposed in this work is a generalization of the one introduced in \cite{maziar1}. This method with some slight improvements using sensitivity analysis of POD basis functions is applied to the Stochastic Burgers equation driven by Brownian motion in \cite{maziar3}. Similar performances are achieved in these papers.
\end{rema}
\section{Concluding remarks}
In this paper, we have proposed a method to enhance the performance of stochastic collocation methods using proper orthogonal decomposition. We have carried out detailed error analyses of the proposed multi-fidelity stochastic collocation methods for parabolic partial differential equations with random forcing terms. We illustrated and supported our theoretical analyses with a numerical example. The analysis of this paper can be simply generalized to parabolic partial differential equations with random initial conditions and random coefficients. Our method only requires a well-posedness argument of the corresponding deterministic equations. Future works in this area can include applications of this method to partial differential equations in fluid mechanics, and proving error estimates for these equations.

\acknowledgements

The authors thank George Karniadakis, Dongbin Xiu, Alireza Doostan, Sergey Lototsky and Peter Kloeden for their valuable comments during the ICERM uncertainty quantification workshop at Brown university. We also thank Karen Willcox for her comments during her visit to George Mason University.

\bibliographystyle{IJ4UQ_Bibliography_Style}

\bibliography{References}
\end{document}